\documentclass[11pt]{article}
\usepackage{amsmath, amssymb, amsfonts, amsthm, graphicx, relsize, mathtools, mathrsfs, euscript, bbm, bbold, float, hyperref, enumitem, url, breakurl, stmaryrd, xcolor}
\usepackage[english]{babel}
\usepackage{titlesec}
\usepackage{tikz}
\titlelabel{\thetitle.\,\,}

\hypersetup{
colorlinks,
linkcolor={black!50!black},
citecolor={black!50!black},
urlcolor={black!80!black}
}

\allowdisplaybreaks

\makeatletter
\newtheorem*{rep@theorem}{\rep@title}
\newcommand{\newreptheorem}[2]{%
\newenvironment{rep#1}[1]{%
 \def\rep@title{#2 \ref{##1}}%
 \begin{rep@theorem}}%
 {\end{rep@theorem}}}
\makeatother

\theoremstyle{plain}
\newtheorem{theorem}{Theorem}[section]
\newreptheorem{theorem}{Theorem}
\newtheorem{lemma}[theorem]{Lemma}
\newtheorem{corollary}[theorem]{Corollary}

\newcommand{\eps}{\varepsilon}
\theoremstyle{definition}
\newtheorem{definition}[theorem]{Definition}
\newtheorem{remark}[theorem]{Remark}
\def\urltilda{\kern-0.15em\lower0.7ex\hbox{\~{}}\kern0.04em}

\definecolor{royalpurple}{RGB}{86,28,168}

\tikzset{
  ep/.style    = {black, thick},
  en/.style    = {royalpurple, thick, dashed},
  epT/.style   = {black, thick, opacity=0.28},
  enT/.style   = {royalpurple, thick, dashed, opacity=0.28},
  vtx/.style   = {circle, draw=black, fill=black, inner sep=1.8pt},
  hubin/.style = {circle, draw=black, fill=black, inner sep=2.4pt},
  hubout/.style= {coordinate}, 
}

\usetikzlibrary{calc,fit,backgrounds}

\begin{document}

\title{\bf{\LARGE Orthogonal signed graphs of degree 5}}


\author{
Maxwell Levit$^{^{1}}$  \qquad  Bojan Mohar$^{^{2}}$  \qquad   Behruz Tayfeh-Rezaie$^{^{3}}$\\[2mm]
$^{^1}$York University\\
$^{^2}$Simon Fraser University\\
$^{^3}$Institute for Research in
Fundamental  Sciences (IPM)\\[2mm]
\href{mailto:maxwell\_levit@sfu.ca}{mlevit@yorku.ca} \qquad
\href{mailto:mohar@sfu.ca}{mohar@sfu.ca}  \qquad
\href{mailto:tayfeh-r@ipm.ir}{tayfeh-r@ipm.ir}}

\date{}

\maketitle

\begin{abstract}
\noindent
An orthogonal signed graph is a connected signed graph whose signed adjacency matrix has pairwise orthogonal rows. They are closely related to Hadamard matrices, maximal arrangements of equiangular lines, bipartite Ramanujan graphs, and the remarkable resolution of the Sensitivity Conjecture. 
Orthogonal signed graphs of degree at most 4 have been completely classified, and partial results were known for degree $5$.
We complete the classification of orthogonal signed graphs of maximum degree $5$. We also provide several new infinite families of 6- and 8-regular orthogonal signed graphs.
 \\[3mm]
\end{abstract}

\section{Introduction}
A {\sl signed graph}  is a pair $\Gamma=(G, \sigma)$, where $G$ is a simple graph,  called {\sl the underlying graph}, and $\sigma$,
 called the {\sl  signature}, is a function that assigns 1 or $-1$  to each edge of $G$. A signed graph is \textit{regular}, \textit{simple}, \textit{connected}, or \textit{bipartite} if its underlying graph is respectively regular, simple, connected, or bipartite.

The {\sl signed adjacency matrix} $A=A(\Gamma)$ of  $\Gamma$ is obtained from the standard adjacency matrix of $G$ by replacing each nonzero entry at position $ij$ with $\sigma(ij)$.
 We say that $\Gamma$  is {\sl orthogonal }  if $G$ is connected and $A^2=D$, where  $D$ is a diagonal matrix with the vertex degrees of $G$ on the diagonal. If instead the signed adjacency matrix $A$ is given, we use $\Gamma(A)$ to denote the signed graph. In Section \ref{sec:prelim} we will show that any orthogonal signed graph is regular.

Orthogonal $k$-regular signed  graphs have been classified  for $k\leqslant4$ in \cite{MK}. See also \cite{Gh,Ho} for other proofs.  For $k=5$, the classification was initiated in \cite{St}, though it was not completed. 

Our main result (Theorem \ref{Thm:Main_thm_technical}) is to construct a new infinite family of $5$-regular orthogonal signed graphs and then prove that the list from \cite{St} is complete once this family is appended.

\begin{theorem} \label{Thm:Main_thm_colloquial}
Every $5$-regular orthogonal signed graph is switching equivalent to one of the following: 

\begin{enumerate}
    \item One of four `sporadic' graphs.
    \item $T_n\square K_2$, $n\geqslant3$,
    \item $T_n^+$, $n\geqslant 5$,
    \item $T_{2n}^{\ast}$, $n\geqslant 2$.
\end{enumerate}
\end{theorem}

The undefined notation is found in Definitions  \ref{defT2n}, \ref{defT+} and \ref{defTast} below and the precise statement of the theorem is given in Theorem \ref{Thm:Main_thm_technical}.

In Section \ref{Sec:Kgeq6} we share some thoughts on the characterization of $6$-regular orthogonal signed graphs. We introduce a new method for constructing infinite families via a  `split and join' operation, and give novel infinite families of 6- and 8-regular orthogonal signed graphs  constructed by iterating this operation.

\subsection{Background and motivation}
Orthogonal signed graphs are the optimal solutions to several well-studied spectral problems in combinatorics. Suppose, in analogy with Hadamard matrices, we wish to maximize the determinant of the signed adjacency matrix of a regular graph. Among all signings, the orthogonal signed graphs are our best choice. Indeed, when the fixed graph is complete bipartite, exactly the Hadamard matrices are recovered as the optimizers. 

If instead we wish to minimize the spectral radius, the answer is again the orthogonal signed graphs. Spectral radius minimization for signed graphs is extremely important, as it is the mechanism by which infinite families of bipartite Ramanujan graphs of a fixed degree were constructed \cite{BL,MSS}. We should note that in this application, one must use signed graphs which adhere to a weaker spectral radius bound of $2\sqrt{k-1}$ rather than the optimal bound afforded by orthogonal signed graphs.  

In fact, the spectral theory of signed graphs dates back to the 1970's and an application to the geometry of equiangular lines in real space \cite{LS73}. The maximum number $n$ of pairwise equiangular lines embedded in $\mathbb{R}^d$ is bounded above by a function in terms of $d$ and the angle between the lines. This bound is tight precisely when a signed complete graph $K_n$ associated with the configuration of lines has exactly two distinct eigenvalues. A particularly beautiful example is the $6$ lines through antipodal pairs of vertices of a regular icosahedron, whose associated signed $K_6$ is orthogonal and will appear later under the name $T_3^+$.

More recently, orthogonal signed hypercubes were crucial to Huang's proof  of the long-standing Sensitivity Conjecture \cite{Hu}. 
Using  Huang's argument, one may  show  that in  every orthogonal $k$-regular signed  graph any induced subgraph on more than half of the vertices has a vertex with degree at least $\sqrt{k}$.

See also the survey \cite{Bel} for an overview of the spectral theory of signed graphs including many questions on signed graphs with two distinct eigenvalues.

\subsection{Preliminaries}\label{sec:prelim}
For vertices $u$ and $v$ of the signed graph $\Gamma=(G,\sigma)$ we let $u\sim v$ denote that $u$ and $v$ are adjacent in $G$. We let $N(u)=N_u=\{v\in V(\Gamma): u\sim v\}$ denote the neighborhood of $u$ in $G$. Let $A=A(\Gamma)$ and let \[\lambda(u,v):=(A^2)_{uv}=\sum_{w\in N_u\cap N_v} \sigma(uw)
\cdot \sigma(vw).\]

$\Gamma$ is orthogonal if and only if  $\lambda(u,v)=0$ for each pair of distinct vertices $u$ and $v$. We note a simple but important consequence of this fact. 

\begin{lemma}\label{Lem:evencommon}
    If $\Gamma=(G,\sigma)$ is an orthogonal signed graph, then each pair of distinct vertices of $G$ has an even number of common neighbors. \qed
\end{lemma}

Now we show that all orthogonal signed graphs are regular. Note that essentially the same argument appears in \cite{CS}. We thank Saieed Akbari for bringing this to our attention.

\begin{lemma}\label{regulareven}
Any orthogonal signed graph is regular.
\end{lemma}

\begin{proof}
Let $A$ be the  signed adjacency matrix of an orthogonal signed graph  $\Gamma$. Let $k_1> k_2>\dots> k_m$ be  all distinct vertex degrees of $\Gamma$.
Multiply  each row  of  $A$ by $1/\sqrt{\ell}$ where $\ell$ is the degree of the vertex corresponding to that row and call the resulting matrix $B$. It follows that $BB^T=B^TB=I$, where
$I$ is the identity matrix.
 Let $X_i$ be the set of vertices of degree $k_i$. Let $v\in X_1$ and suppose that $v$ has $x_i$ neighbors in $X_i$. Then, from  $B^TB=I$ and looking at the column corresponding to $v$ in $B$,
  we have
  \begin{align*}
 1&=\frac{x_1}{k_1}+\frac{x_2}{k_2}+\cdots+\frac{x_m}{k_m}\\
 &\geqslant\frac{x_1}{k_1}+\frac{x_2}{k_1}+\cdots+\frac{x_m}{k_1}\\
 &=\frac{x_1+x_2+\cdots+x_m}{k_1}\\
 &=1,
 \end{align*}
which means that  $x_2=x_3=\cdots=x_m=0$. Since the underlying graph of $\Gamma$ is connected, we find that  $X_2=X_3=\cdots=X_m=\varnothing$. Hence,
$\Gamma$ is $k_1$-regular.
\end{proof}

Two signed graphs $\Gamma$ and $\Gamma'$ are called {\sl switching equivalent} if there is a signed permutation
matrix $P$  such that $$A(\Gamma')=P^TA(\Gamma)P.$$
In particular, if we change the sign of all edges incident with a fixed vertex $v$ of  $\Gamma$, then we obtain a switching equivalent copy of $\Gamma$.
We call this operation {\sl switching } at $v$. If $u$ is some vertex of a signed graph, by switching at each neighbor $v$ of $u$ such that $\sigma(u,v)=-1$, we obtain an equivalent signature in which each edge incident with $u$ is positive. In this case we say that the signature is {\sl normalized} at $u$.

\subsection{Weighing matrices}

A {\it   weighing} matrix $W$  of order $n$ and of weight $k$  is a $(-1,0,1)$ square matrix of  order $n$ such that $WW^T=kI$, where $W^T$ is
the transpose of $W$ and $I$ is the identity matrix.
Note that for  any   orthogonal signed graph  $\Gamma$,  $A(\Gamma)$  is a  weighing matrix with the extra conditions that it is symmetric and  zero diagonal.
A weighing matrix of weight $n$ is a Hadamard matrix and a weighing matrix of weight $n-1$ is a conference matrix.
It is well known that if $n\equiv 2 \pmod{4}$,  then any  conference matrix of order $n$ can be made  symmetric and zero diagonal.
Therefore, when $n\equiv 2 \pmod{4}$, orthogonal  signed graphs of the complete graph  $K_n$ are equivalent to conference matrices  of  order $n$.
Via the following construction, orthogonal  signings of the complete bipartite graph  $K_{n,n}$ are equivalent to   Hadamard matrices  of  order $n$.
\begin{definition}\label{defBiW}
For a weighing matrix $W$, we define the {\sl bipartite signed graph of $W$} as the orthogonal signed graph with signed adjacency matrix
$$\begin{bmatrix}
0 & W\\
W^T & 0
\end{bmatrix}.$$
\end{definition}

Two particular weighing matrices are of note. They are denoted $W(7,4)$ and $W(12,5)$ in the literature of weighing matrices, and their bipartite signed graphs will appear in our classification. The underlying graph associated with $W(7,4)$ is the co-Heawood graph: The graph with vertices the points and lines of the Fano plane and adjacency between non-incident point-line pairs.
The underlying graph associated with $W(12,5)$ is the bipartite double of the icosahedron. 


\subsection{Products of signed graphs}\label{Sec:Prod}

The following theorem (whose proof is simply matrix multiplication) may be used recursively to construct infinite families of orthogonal  signed graphs of unbounded degree. It generalizes Theorems  1.2 and 1.3 from \cite{Hong}.
\begin{theorem}\label{recur}
Let $A$ be the signed adjacency matrix of an orthogonal $k$-regular signed  graph on $n$ vertices,  $B$ be  a   weighing matrix of order $n$ and  of  weight $r$  such that $AB=BA$  and  $C$ be  a   weighing matrix of order $m$ and  of  weight $s$.
Then
$$\begin{bmatrix}
I_m\otimes A(\Gamma) & C\otimes B\\
C^T\otimes B^T & -I_m\otimes A(\Gamma)
\end{bmatrix}$$
is  signed adjacency matrix of an orthogonal $(k+rs)$-regular signed  graph on $2mn$ vertices. \qed
\end{theorem}

Letting $B=I$ above, we obtain the following.
\begin{corollary}[\cite{Hong}]\label{cartesian}
If $\Gamma$ is an orthogonal  $k$-regular signed  graph on $m$ vertices   and  $\Gamma'$ is an orthogonal $\ell$-regular bipartite signed     graph  on  $n$ vertices with signed adjacency matrix
$\begin{bmatrix}
0 & C\\
C^T & 0
\end{bmatrix}$, then
$$\begin{bmatrix}
I_n\otimes A(\Gamma) & C\otimes I_m\\
C^T\otimes I_m & -I_n\otimes A(\Gamma)
\end{bmatrix}$$
is  signed adjacency matrix of an orthogonal $(k+\ell)$-regular signed  graph on $2mn$ vertices, denoted by $\Gamma\square \Gamma'$. \qed
\end{corollary}

A simple but important application of this corollary is when  $\Gamma'=K_2$.  We let $\Gamma \square K_2$ denote the signed graph obtained in this case and note that its underlying graph is the Cartesian product $G\square K_2$. For example, starting with $\Gamma=K_1$ and applying this construction recursively $n$ times, we construct the orthogonal signed $n$-dimensional hypercube
which we denote by $\overline{Q_n}$.

The next lemma shows that this Cartesian product with $K_2$ is the only way for a $(k-1)$-regular orthogonal signed graph to occur as an induced subgraph of a $k$-regular orthogonal signed graph. It will be used shortly in our classification of $5$-regular orthogonal signed graphs.

\begin{lemma} \label{Lem:k-1_reg_ind}
    Suppose $\Gamma$ is a $k$-regular orthogonal signed graph which contains as a subgraph the $(k-1)$-regular orthogonal signed graph $\Delta$. Then either $V(\Gamma)=V(\Delta)$ or $\Gamma$ is switching equivalent to $\Delta\square K_2$.
\end{lemma}

\begin{proof}
    Let $X=V(\Delta)$ and suppose $Y=V(\Gamma)\setminus X$ is nonempty. First note that $E(X,Y)$ is a matching. Indeed, since $\Delta$ is $(k-1)$-regular it is clear that no two edges of $E(X,Y)$ are incident in $X$. If two edges of $E(X,Y)$ were incident in $Y$, then their ends in $X$ would have an odd number of common neighbors in $\Gamma$, a contradiction.

    Since $\Gamma$ is connected and $Y$ is nonempty, there exists $u\in X$ with a neighbor $u'\in Y$. For each $v\in X$ such that $v\sim u$, there must exist at least one more 2-walk from $v$ to $u'$ to balance with the 2-walk through $u$. Since $E(X,Y)$ is a matching, the midpoint of such a 2-walk must be a vertex in $Y$; call it $v'$. Moreover $(u,v,v',u')$ is a 4-cycle. Since $\Delta$ is connected, this implies that for each vertex $u\in X$ there is a unique vertex $u'\in Y$ with $u\sim u'$. Moreover, for each edge $u\sim v$ of $\Delta$ there is an edge $u'\sim v'$ in $\Gamma[Y]$. It follows that the subgraph induced by $\{v'\in Y: v\in X\}$ is $(k-1)$-regular and isomorphic to the subgraph induced by $\Gamma[X]$. In other words, $G=\Delta\square K_2$.

    By switching so that each edge of $E(X,Y)$ is positive, we construct a switching equivalent signed graph in which the signature of $\Delta$ must be the opposite of the signature of $\Gamma[Y]$, hence $\Gamma$ is switching equivalent to $\Delta\square K_2$.
\end{proof}

\subsection{Folded cubes}

The folded $n$-dimensional hypercube $Q_n^{+}$ is obtained from the $n$-dimensional hypercube $Q_n$ by adding an edge between
every pair of vertices at maximal distance.  Here is another example of orthogonal signed graphs.

\begin{theorem}[Alon and Zheng \cite{Al}]\label{folded}
There exists an orthogonal signed graph with underlying graph  $Q_n^+$ if and only if $n\equiv 0,3  \pmod{4}$. \qed
\end{theorem}

We denote the orthogonal signed  folded $n$-dimensional  hypercube given in \cite{Al} by $\overline{Q_n^+}$.


\subsection{The $k=4$ classification}

Keeping in mind the product constructions from Section \ref{Sec:Prod}, it is clear that orthogonal signed graphs of unbounded valency are abundant. The remaining interesting question is the existence and abundance of infinite families of a fixed degree, which brings us to  classification theorems. 

For each $k\leqslant3$, there is a unique orthogonal signed graph.  These are  $\overline{Q_1}$, $\overline{Q_2}$ and $\overline{Q_3}$ as defined above.
The case  $k=4$ was dealt with in \cite{MK}. Other proofs are also given in \cite{Gh,Ho}. To present the result, we need to define the following family of graphs which will also be relevant for our $k=5$ classification.

\begin{definition}\label{defT2n} (cf.~Figure \ref{fig:T,TP,TS}).
We define the  family of orthogonal $4$-regular signed graphs $T_{n}$ for $n\geqslant 3$. The vertex set is $V=\{0,1\}\times\{0,1,\ldots,n-1\}$.
 Each vertex $(0,j)\in V$ is joined to
$(0,j+1)$ and  $(1,j+1)$ with sign 1.  Also,   each vertex $(1,j)\in V$ is joined to
$(0,j+1)$ and  $(1,j+1)$ with sign $-1$.
All computations are modulo $n$.
It is easy to check that $T_{n}$ is in fact orthogonal.
\end{definition}

\begin{figure}
    \centering
\begin{tikzpicture}[scale=1.0]
  \begin{scope}[xshift=0.0cm]
    \node[vtx] (T0) at (0.000,1.150) {};
    \node[vtx] (T1) at (-0.813,0.813) {};
    \node[vtx] (T2) at (-1.150,0.000) {};
    \node[vtx] (T3) at (-0.813,-0.813) {};
    \node[vtx] (T4) at (-0.000,-1.150) {};
    \node[vtx] (T5) at (0.813,-0.813) {};
    \node[vtx] (T6) at (1.150,-0.000) {};
    \node[vtx] (T7) at (0.813,0.813) {};
    \node[vtx] (T8) at (0.000,2.000) {};
    \node[vtx] (T9) at (-1.414,1.414) {};
    \node[vtx] (T10) at (-2.000,0.000) {};
    \node[vtx] (T11) at (-1.414,-1.414) {};
    \node[vtx] (T12) at (-0.000,-2.000) {};
    \node[vtx] (T13) at (1.414,-1.414) {};
    \node[vtx] (T14) at (2.000,-0.000) {};
    \node[vtx] (T15) at (1.414,1.414) {};
    \draw[en] (T0) -- (T1);
    \draw[en] (T0) -- (T7);
    \draw[en] (T0) -- (T9);
    \draw[ep] (T0) -- (T15);
    \draw[en] (T1) -- (T2);
    \draw[ep] (T1) -- (T8);
    \draw[en] (T1) -- (T10);
    \draw[en] (T2) -- (T3);
    \draw[ep] (T2) -- (T9);
    \draw[en] (T2) -- (T11);
    \draw[en] (T3) -- (T4);
    \draw[ep] (T3) -- (T10);
    \draw[en] (T3) -- (T12);
    \draw[en] (T4) -- (T5);
    \draw[ep] (T4) -- (T11);
    \draw[en] (T4) -- (T13);
    \draw[en] (T5) -- (T6);
    \draw[ep] (T5) -- (T12);
    \draw[en] (T5) -- (T14);
    \draw[en] (T6) -- (T7);
    \draw[ep] (T6) -- (T13);
    \draw[en] (T6) -- (T15);
    \draw[en] (T7) -- (T8);
    \draw[ep] (T7) -- (T14);
    \draw[ep] (T8) -- (T9);
    \draw[ep] (T8) -- (T15);
    \draw[ep] (T9) -- (T10);
    \draw[ep] (T10) -- (T11);
    \draw[ep] (T11) -- (T12);
    \draw[ep] (T12) -- (T13);
    \draw[ep] (T13) -- (T14);
    \draw[ep] (T14) -- (T15);
    \node at (0,-2.7) {\large $T_8$};
  \end{scope}
  \begin{scope}[xshift=5.6cm]
    \node[vtx] (P0) at (0.000,1.150) {};
    \node[vtx] (P1) at (-0.813,0.813) {};
    \node[vtx] (P2) at (-1.150,0.000) {};
    \node[vtx] (P3) at (-0.813,-0.813) {};
    \node[vtx] (P4) at (-0.000,-1.150) {};
    \node[vtx] (P5) at (0.813,-0.813) {};
    \node[vtx] (P6) at (1.150,-0.000) {};
    \node[vtx] (P7) at (0.813,0.813) {};
    \node[vtx] (P8) at (0.000,2.000) {};
    \node[vtx] (P9) at (-1.414,1.414) {};
    \node[vtx] (P10) at (-2.000,0.000) {};
    \node[vtx] (P11) at (-1.414,-1.414) {};
    \node[vtx] (P12) at (-0.000,-2.000) {};
    \node[vtx] (P13) at (1.414,-1.414) {};
    \node[vtx] (P14) at (2.000,-0.000) {};
    \node[vtx] (P15) at (1.414,1.414) {};
    \draw[en] (P0) -- (P1);
    \draw[en] (P0) -- (P7);
    \draw[ep] (P0) -- (P8);
    \draw[en] (P0) -- (P9);
    \draw[ep] (P0) -- (P15);
    \draw[en] (P1) -- (P2);
    \draw[ep] (P1) -- (P8);
    \draw[ep] (P1) -- (P9);
    \draw[en] (P1) -- (P10);
    \draw[en] (P2) -- (P3);
    \draw[ep] (P2) -- (P9);
    \draw[ep] (P2) -- (P10);
    \draw[en] (P2) -- (P11);
    \draw[en] (P3) -- (P4);
    \draw[ep] (P3) -- (P10);
    \draw[ep] (P3) -- (P11);
    \draw[en] (P3) -- (P12);
    \draw[en] (P4) -- (P5);
    \draw[ep] (P4) -- (P11);
    \draw[ep] (P4) -- (P12);
    \draw[en] (P4) -- (P13);
    \draw[en] (P5) -- (P6);
    \draw[ep] (P5) -- (P12);
    \draw[ep] (P5) -- (P13);
    \draw[en] (P5) -- (P14);
    \draw[en] (P6) -- (P7);
    \draw[ep] (P6) -- (P13);
    \draw[ep] (P6) -- (P14);
    \draw[en] (P6) -- (P15);
    \draw[en] (P7) -- (P8);
    \draw[ep] (P7) -- (P14);
    \draw[ep] (P7) -- (P15);
    \draw[ep] (P8) -- (P9);
    \draw[ep] (P8) -- (P15);
    \draw[ep] (P9) -- (P10);
    \draw[ep] (P10) -- (P11);
    \draw[ep] (P11) -- (P12);
    \draw[ep] (P12) -- (P13);
    \draw[ep] (P13) -- (P14);
    \draw[ep] (P14) -- (P15);
    \node at (0,-2.7) {\large $T_8^+$};
  \end{scope}
  \begin{scope}[xshift=11.2cm]
    \node[vtx] (S0) at (0.000,1.150) {};
    \node[vtx] (S1) at (-0.813,0.813) {};
    \node[vtx] (S2) at (-1.150,0.000) {};
    \node[vtx] (S3) at (-0.813,-0.813) {};
    \node[vtx] (S4) at (-0.000,-1.150) {};
    \node[vtx] (S5) at (0.813,-0.813) {};
    \node[vtx] (S6) at (1.150,-0.000) {};
    \node[vtx] (S7) at (0.813,0.813) {};
    \node[vtx] (S8) at (0.000,2.000) {};
    \node[vtx] (S9) at (-1.414,1.414) {};
    \node[vtx] (S10) at (-2.000,0.000) {};
    \node[vtx] (S11) at (-1.414,-1.414) {};
    \node[vtx] (S12) at (-0.000,-2.000) {};
    \node[vtx] (S13) at (1.414,-1.414) {};
    \node[vtx] (S14) at (2.000,-0.000) {};
    \node[vtx] (S15) at (1.414,1.414) {};
    \draw[en] (S0) -- (S1);
    \draw[en] (S0) -- (S7);
    \draw[en] (S0) -- (S9);
    \draw[ep] (S0) to[bend left=32] (S12);
    \draw[ep] (S0) -- (S15);
    \draw[en] (S1) -- (S2);
    \draw[ep] (S1) -- (S8);
    \draw[en] (S1) -- (S10);
    \draw[ep] (S1) to[bend left=32] (S13);
    \draw[en] (S2) -- (S3);
    \draw[ep] (S2) -- (S9);
    \draw[en] (S2) -- (S11);
    \draw[ep] (S2) to[bend left=32] (S14);
    \draw[en] (S3) -- (S4);
    \draw[ep] (S3) -- (S10);
    \draw[en] (S3) -- (S12);
    \draw[ep] (S3) to[bend left=32] (S15);
    \draw[en] (S4) -- (S5);
    \draw[ep] (S4) to[bend left=32] (S8);
    \draw[ep] (S4) -- (S11);
    \draw[en] (S4) -- (S13);
    \draw[en] (S5) -- (S6);
    \draw[ep] (S5) to[bend left=32] (S9);
    \draw[ep] (S5) -- (S12);
    \draw[en] (S5) -- (S14);
    \draw[en] (S6) -- (S7);
    \draw[ep] (S6) to[bend left=32] (S10);
    \draw[ep] (S6) -- (S13);
    \draw[en] (S6) -- (S15);
    \draw[en] (S7) -- (S8);
    \draw[ep] (S7) to[bend left=32] (S11);
    \draw[ep] (S7) -- (S14);
    \draw[ep] (S8) -- (S9);
    \draw[ep] (S8) -- (S15);
    \draw[ep] (S9) -- (S10);
    \draw[ep] (S10) -- (S11);
    \draw[ep] (S11) -- (S12);
    \draw[ep] (S12) -- (S13);
    \draw[ep] (S13) -- (S14);
    \draw[ep] (S14) -- (S15);
    \node at (0,-2.7) {\large $T_8^{\ast}$};
  \end{scope}
    \end{tikzpicture}
    \caption{An $4$-regular orthogonal signed graph (left) and two orthogonal $5$-regular signed graphs.}
    \label{fig:T,TP,TS}
\end{figure}
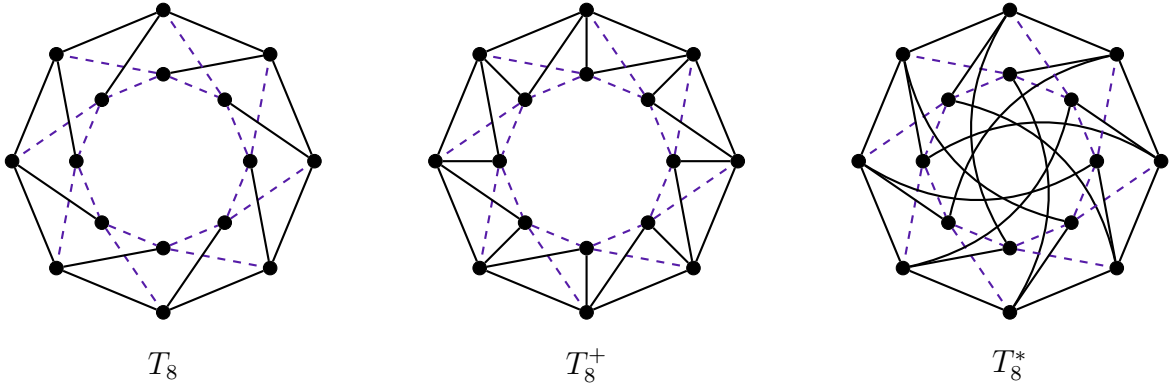

\begin{remark}\label{remT2n}
Let $\Gamma=T_{n}$. Let us interchange the labels of vertices $(0,j)$ and $(1,j)$ in $\Gamma$. Let us do this also for the pair $(0,j+1)$ and $(1,j+1)$.
Then we do switching  at both vertices $(0,j+1)$ and $(1,j+1)$.  We find that  $\Gamma=T_{n}$, unchanged. This property of  $T_{n}$ will be used later.
\end{remark}

\begin{theorem}[\cite{Gh,Ho,MK}]\label{4reg}
Any orthogonal $4$-regular signed graph is switching equivalent to one of the following.
\begin{itemize}
\item[$(1)$] $\overline{Q_4}$ (signed hypercube),
\item[$(2)$] $F$, the bipartite signed graph of $W(7,4)$ (see Definition \ref{defBiW}),
\item[$(3)$] $T_{n}$ for $n\geqslant 3$  (see Definition \ref{defT2n}).
\end{itemize}
\end{theorem}

\section{Orthogonal $5$-regular signed graphs}\label{Sec:k=5}
In this section  we classify orthogonal $5$-regular signed graphs up to switching equivalence.
In \cite{St}, a characterization, but not a complete classification,  of such signed graphs is given. The characterization there relies on a previous classification of weighing matrices of weight 5 \cite{CRS86,HM12}. Our proof does not rely on this classification of weighing matrices. So one can obtain a new proof of the classification of weight 5 weighing matrices from our result.

More importantly, we identify a new infinite family of examples of orthogonal signed graphs (Definition \ref{defTast}) which were, to our knowledge, not previously known. We begin by defining the pertinent infinite families.

\begin{definition}\label{defT+} (cf.~Figure \ref{fig:T,TP,TS}).
We define the  family of orthogonal $5$-regular signed graphs $T^+_{n}$ for $n\geqslant 3$. Start with $T_{n}$ and then add a perfect matching by joining
each vertex $(0,j)$  to
$(1,j)$ with sign $1$. 
\end{definition}

\begin{definition}\label{defTast} (cf.~Figure \ref{fig:T,TP,TS}).
We define the  family of orthogonal $5$-regular signed graphs $T^{\ast}_{2n}$ for $n\geqslant 2$. Start with $T_{2n}$ and  then  add a perfect matching as follows.
For $j=0,1,\ldots,n-1$, join $(0,j)$  to $(1,j+n)$ and $(1,j)$  to $(0,j+n)$ both with sign 1.
\end{definition}

\begin{lemma}\label{remcases}
The only isomorphism between elements of the three infinite families  $T_{n}\square K_2,T^+_{n},T^{\ast}_{2n}$ is  $T^+_{4}\simeq T^{\ast}_{4}$. 
\end{lemma}

\begin{proof}
First note that $T^{\ast}_{2n}$ has a signed subgraph $T_{2n}$ while $T_{n}\square K_2$ has no such signed subgraph, so these two families are disjoint. Next note that $T^+_{n}$ has triangles, but $T_{n}\square K_2$ has no triangle except for $n=3$, and in $T^+_6$ each vertex neighborhood induces two triangles sharing a single vertex while in $T_3\square K_2$ each vertex neighborhood induces $C_4\square K_1$. Finally, $T^{\ast}_{2n}$  has
triangles if and only if $n=2$. Here one can check that in fact $T^+_{4}$ is switching equivalent to $T^{\ast}_{4}$.
\end{proof}

Let $D_G$ denote the graph on $V(G)$ with $u$ and $v$ adjacent if and only if they have exactly four common neighbors in $G$. When $\Gamma=(G,\sigma)$ is a $5$-regular orthogonal signed graph, certain properties of $D_G$ are immediate. For instance, it is a disjoint union of complete graphs. It is convenient to organize our classification based on the structure of $D_G$.

\begin{theorem} \label{Thm:Main_thm_technical}
    Let $\Gamma=(G,\sigma)$ be a $5$-regular orthogonal signed graph. We have the following classification up to switching equivalence.

    \begin{enumerate}
        \item $D_G$ is edgeless  and $\Gamma$ is one of four signed graphs.

        \begin{enumerate}
        
\item[$(1.1)$] $\overline{Q_4^+}$ (signed folded hypercube, see  Theorem  \ref{folded}),
\item[$(1.2)$] Bipartite signed graph of $W(12,5)$ (see Definition \ref{defBiW}),
\item[$(1.3)$] $F\square K_2$ (see Theorem \ref{4reg} (2)),
\item[$(1.4)$] $\overline{Q_5}$ (signed hypercube).
        
    \end{enumerate}
    \item $D_G$ contains a triangle and $\Gamma$ is one of four signed graphs. 

    \begin{enumerate}
        \item[$(2.1)$] $T_3^+$,
        \item[$(2.2)$] $T_6^{\ast}$,
        \item[$(2.3)$] $T_4\square K_2$,
        \item[$(2.4)$] $T^+_4$.

    \end{enumerate}
        
    \item $D_G$ is a (nonempty) matching and $\Gamma$ belongs to one of three infinite families.
        \begin{enumerate}
            \item[$(3.1)$] $T_n\square K_2$, $n\geqslant3,n\neq 4$,
            \item[$(3.2)$] $T_n^+$, $n\geqslant 5$,
            \item[$(3.3)$] $T_{2n}^{\ast}$, $n\geqslant 4$.
        \end{enumerate}

        \end{enumerate}
\end{theorem}

In the remainder of this section we prove Theorem \ref{Thm:Main_thm_technical} through a series of lemmas. First we describe the local structure of a $5$-regular orthogonal signed graph. Below we will use $G_u=G[N_u]$ to denote the subgraph of $G$ induced on the neighborhood of $u$.

\begin{lemma} \label{Lem:k=5_Nbhd}
    If $\Gamma$ is a $5$-regular orthogonal signed graph and $u,v$ are distinct vertices, then $u$ and $v$ have either $0$, $2$, or $4$ common neighbors. Moreover, the subgraph $\Gamma_u=\Gamma[N_u]$ induced by the neighborhood of $u$ is switching equivalent to one of the four signed graphs in Figure \ref{Fig:5nbhds}.
\end{lemma}

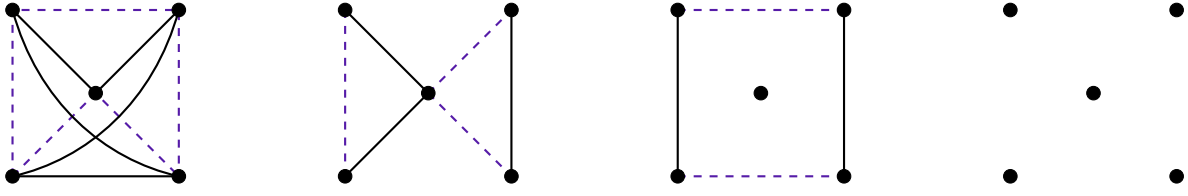
\begin{figure}
    \centering
    \begin{tikzpicture}[x=1.1cm,y=1.1cm]

\begin{scope}[shift={(0,0)}]
  \coordinate (A) at (0,2);
  \coordinate (B) at (2,2);
  \coordinate (C) at (2,0);
  \coordinate (D) at (0,0);
  \coordinate (E) at (1,1);

  \draw[en] (D) -- (A) -- (B) -- (C);
  \draw[en] (D) -- (E) -- (C);

  \draw[ep] (D) -- (C);
  \draw[ep] (A) -- (E) -- (B);
  \draw[ep] (A) .. controls (0.3,0.9) and (1.1,0.2) .. (C);
  \draw[ep] (B) .. controls (1.7,0.9) and (0.9,0.2) .. (D);

  \foreach \P in {A,B,C,D,E}{\node[vtx] at (\P) {};}
\end{scope}

\begin{scope}[shift={(4,0)}]
  \coordinate (A) at (0,2);
  \coordinate (B) at (2,2);
  \coordinate (C) at (2,0);
  \coordinate (D) at (0,0);
  \coordinate (E) at (1,1);

  \draw[en] (A) -- (D);
  \draw[ep] (B) -- (C);

  \draw[ep] (A) -- (E) -- (D);
  \draw[en] (E) -- (B);
  \draw[en] (E) -- (C);

  \foreach \P in {A,B,C,D,E}{\node[vtx] at (\P) {};}
\end{scope}

\begin{scope}[shift={(8,0)}]
  \coordinate (A) at (0,2);
  \coordinate (B) at (2,2);
  \coordinate (C) at (2,0);
  \coordinate (D) at (0,0);
  \coordinate (E) at (1,1);

  \draw[en] (A) -- (B);
  \draw[en] (D) -- (C);
  \draw[ep] (A) -- (D);
  \draw[ep] (B) -- (C);

  \foreach \P in {A,B,C,D,E}{\node[vtx] at (\P) {};}
\end{scope}

\begin{scope}[shift={(12,0)}]
  \coordinate (A) at (0,2);
  \coordinate (B) at (2,2);
  \coordinate (C) at (2,0);
  \coordinate (D) at (0,0);
  \coordinate (E) at (1,1);
  \foreach \P in {A,B,C,D,E}{\node[vtx] at (\P) {};}
\end{scope}

\end{tikzpicture}

    \caption{The four possible neighborhoods of a vertex in a 5-regular orthogonal signed graph, up to switching equivalence.}
    \label{Fig:5nbhds}
\end{figure}

\begin{proof}
The vertices $u$ and $v$ must have $0$, $2$, or $4$ common neighbors by Lemma \ref{Lem:evencommon}. Now for any vertex $u$ and any $v\in N_u$ there are either $0$, $2$, or $4$ other vertices of $N_u$ incident with $v$. So $G_u=G[N_u]$ is an Eulerian graph on $5$ vertices. Moreover, we may assume via switching that each edge incident with $u$ has positive sign. Thus exactly half the edges of $G_u$ incident with $v$ are positive and half are negative. This rules out several Eulerian graphs on $5$ vertices, e.g. the $5$-cycle. The four listed possibilities survive.
\end{proof}

Two of these four local subgraphs lead only to a single signed graph each.

\begin{lemma} \label{Lem:Locally_k5_c4}
If $\Gamma=(G,\sigma)$ is a $5$-regular orthogonal signed graph containing a vertex $u$ such that $G_u$ is isomorphic to $K_5$ or $C_4\sqcup K_1$ respectively, then $\Gamma=T_3^+$ or $\Gamma=T_4\square K_2$ respectively. In each case there is a unique signature up to switching.
\end{lemma}

\begin{proof}
    If $G_u\cong K_5$ then $G\cong K_6$. Lemma \ref{Lem:k=5_Nbhd} implies there is only one signature in this case, namely $T_3^+$.

    Now suppose $G_u\cong C_4\sqcup K_1$. We label the $C_4$ cyclically with $a,b,c,d$ and we let the fifth vertex be $u'$. We may assume $\sigma$ is normalized on $u$ and $\sigma(ab)=\sigma(cd)=-1$ and $\sigma(bc)=\sigma(ad)=1$.

    In the subgraph $G[V(G_u)\cup \{u\}]$ the vertices $a,c$ are joined by two negative paths of length 2 and one positive path, so there must be some distinct vertex $x$ incident to both $a$ and $c$ such that $\sigma(ax)=\sigma(cx)$. Similarly there is a vertex $y$ incident with both $b$ and $d$ such that $\sigma(by)=\sigma(dy)$. 

    If $x=y$ then $G$ contains an induced $K_{2,2,2}$ and Lemma \ref{Lem:k-1_reg_ind} implies $G$ is $K_{2,2,2}\square K_2$ with a unique signature up to switching, namely $T_4\square K_2$. 

    If $x\neq y$ then $u,x$ currently have two common neighbors but both the paths from $u$ to $x$ through those neighbors have the same sign, so $u$ and $x$ must have two additional common neighbors from among $\{b,d,u'\}.$ By symmetry we may assume $x\sim d$. Now $G_a$ contains the path $(x,d,u,b)$ hence $G_a$ must be $C_4\sqcup K_1$ and $x\sim b$. 

    By a similar argument we conclude $y\sim a$ and $y\sim c$. So $u,x,y$ each have $a,b,c,d$ as neighbors. Since $\sigma(ax)=\sigma(cx)$ and $u$ is normalized, we must have $\sigma(bx)=\sigma(dx)=-\sigma(ax)$ and similarly for the edges incident with $y$. But then $\lambda(x,y)=\pm 4$, a contradiction. 
\end{proof}

Now we will study the cases in Theorem \ref{Thm:Main_thm_technical}. Note that the underlying graphs in Case 1 are  so-called  $(0,2)$-graph. These are connected graphs in which any two vertices have either $0$ or $2$ common neighbors.
It is known that
 $(0,2)$-graphs are regular, say of valency $k$, and have at most $2^k$ vertices \cite{Mu}. So there are only a modest number of candidates to check for an orthogonal signature \cite{Br,BrCh}. 
 Out of eight candidate $5$-regular graphs, exactly the listed $4$ admit orthogonal signatures. 
 
\begin{lemma} \label{Lem:MT_02_graph}
    If we are in Case 1 of Theorem \ref{Thm:Main_thm_technical} then $\Gamma$ is one of the four listed signed graphs. 
\end{lemma}

\begin{lemma}\label{Lem:MT_D4triangle}
    If we are in Case 2 of Theorem \ref{Thm:Main_thm_technical} then $\Gamma$ is one of the four listed signed graphs.
\end{lemma}

\begin{proof}
    
Let $a,b,c$ be three vertices such that each pair has four common neighbors in $G$. Note that \[5=|N_a|\geqslant |N_a\cap N_b|+|N_a\cap N_c|-|N_a\cap N_b\cap N_c|=8-|N_a\cap N_b\cap N_c|\]

So there are three vertices $x,y,z$ adjacent to each of $a,b,c$. We proceed with cases based on the adjacency (in $G$) of $a,b,c$.

Case 1: $G[\{a,b,c\}]$ contains $P_3$. Say $a\sim b\sim c$. The induced subgraph $G_b$ contains $K_{2,3}$ so Lemma \ref{Lem:k=5_Nbhd} implies $G_b=K_5$, hence $G=K_6$.

Case 2: $G[\{a,b,c\}]=K_2\sqcup K_1$. Say $a\sim b$. The fourth common neighbor of $a,b$ is some $w\notin \{c,x,y,z\}$ and $G_a$ contains $K_{1,4}$ on $\{b;w,x,y,z\}$ hence Lemma \ref{Lem:k=5_Nbhd} implies $G_a$ is $C_3\lor C_3$. Without loss we assume $w\sim x$ and $y\sim z$. Now $y$ and $x$ require one further common neighbor which must be some new vertex $d$. So, in the induced subgraph $G_y$ the vertex $z$ has degree at least 3 and Lemma \ref{Lem:k=5_Nbhd} implies $d\sim z$ and $d\sim c$. A final application of the lemma shows that $c\sim w$ and the graph is $K_{4,4}+M$. 

Case 3: $G[\{a,b,c\}]=3K_1$ and $|N_a\cap N_b\cap N_c|=4$. Say $N_a\cap N_b\cap N_c=\{w,x,y,z\}$. If $w\sim x$ then Lemma \ref{Lem:k=5_Nbhd} applied to $G_x$ would require an adjacency among $a,b,c$ and we are not in this case, hence $G$ contains an induced $K_{3,4}$. Up to switching equivalence the corresponding submatrix of $A(\Gamma)$ is \[\begin{array}{cccccc}
    1&1&1&1\\1&1&-&-\\ 1&-&1&-\\
    \end{array}\]

So the fourth common neighbor $d$ of $w,x$ must in fact have four common neighbors with each of $a,b,c$ and $G$ contains an orthogonal signed $K_{4,4}$ subgraph. Now Lemma \ref{Lem:k-1_reg_ind} implies $G=K_{4,4}\square K_2$.

Case 4: $G[\{a,b,c\}]=3K_1$ and $|N_a\cap N_b\cap N_c|=3$. Each of the three pairs among $a,b,c$ require an additional common neighbor $\alpha,\beta,\gamma$. We claim that $S=\{x,y,z,\alpha,\beta,\gamma\}$ induces an independent set in $G$. Indeed, if $e$ were an edge with ends in $S$ then Lemma \ref{Lem:k=5_Nbhd} applied to one of $G_a,G_b,G_c$ implies that $S$ contains a $C_{3}\lor C_3$ subgraph. But the degree of each vertex of $S$ is already at least $2$ without considering edges within $S$, so this is impossible. 

It follows that, up to switching, the submatrix of the signed adjacency matrix of $G$ with rows indexed by $\{a,b,c\}$ and columns indexed by $S$ is 
    \[\begin{array}{cccccc}
    1&1&1&1&1&0\\1&1&-&-&0&1\\ 1&-&1&0&-&1\\
    \end{array}\]

Similar to the previous argument, there is  a unique way, up to switching, to complete this matrix to a $6\times 6$ orthogonal submatrix (a weighing matrix of weight 5), here $G$ is the bipartite double cover of $K_6$.
\end{proof}


\begin{lemma}\label{Lem:Tn_subgraph}
    If we are in Case 3 of Theorem \ref{Thm:Main_thm_technical} then $\Gamma$ contains a signed subgraph switching equivalent to $T_n$ for some $n\geqslant 3$.
\end{lemma}

\begin{proof}
Let $G$ be as in Case 3 of the theorem. We define a \textit{chain} $C$ in $G$ to be a subgraph of $G$ in which all vertices have degree $2$ or $4$. Now let $u_1,v_1$ be a pair of vertices in $G$ with four common neighbors. We consider the maximal chain of $G$ containing $u_1,v_1$. Of the six $4$-cycles formed by $u_1,v_1$ and their common neighbors, exactly two are positive. Hence we may label the common neighbors as $u_2,v_2,u_k,v_k$ so that the $4$-cycle $(u_1,u_i,v_1,v_j)$ is positive if and only if $i=j$. 

Clearly the maximal chain under consideration must contain each of $u_2,v_2,u_k,v_k$. Since the cycle $(u_1,u_2,v_1,v_2)$ is positive, the vertices $u_2$ and $v_2$ must have four common neighbors in $G$, hence the maximal chain must contain vertices $u_3,v_3$ adjacent to both $u_2$ and $v_2$, and the cycle $(u_2,v_2,u_3,v_3)$ must also be positive, so this cannot be the whole chain either.

We claim that our maximal chain is in fact the desired $T_n$ subgraph. Indeed, no $u_i$ and $v_i$ may coincide with $u_j$ and $v_j$ for $j<i$ since this would yield a vertex of degree six. So maximality forces the chain to close, with some $u_i=u_k$, but then $u_k$ must have four common neighbors with both $v_k$ and $v_i$, so by the assumption of the case, $v_k=v_i$.
\end{proof}

\begin{lemma}
    If we are in Case 3 of Theorem \ref{Thm:Main_thm_technical} then $G$ belongs to one of the three listed infinite families. 
\end{lemma}

\begin{proof}
By Lemma \ref{Lem:Tn_subgraph} the signed graph contains $T_n$ as a subgraph for some $n\geqslant 3.$ Let $X$ be the vertex  set of such a signed subgraph and let $\overline{X}$ be the set of
remaining vertices of $\Gamma$.
We  observe  that  $N_u\cap N_v\subseteq  X$ for  any two distinct vertices  $u,v\in X$. Since otherwise, $|N_u\cap N_v|$ would be odd, impossible.
This shows that $\Gamma[X]$ is an  orthogonal signed graph and so is regular.

If  $\Gamma[X]$ is 4-regular, then $T_n$ is an induced subgraph and Lemma \ref{Lem:k-1_reg_ind} implies that $G$ is $T_n\square K_2$.

So we assume  $\Gamma[X]$ is 5-regular, hence spanning.
We may let  $X=\{0,1\}\times\{0,1,\ldots,n-1\}$, where we use modulo $n$ computation on the second coordinate of elements of $X$.
 For simplicity, let us  write  $u_j=(0,j)$ and $v_j=(1,j)$ for  $0\leqslant j \leqslant n-1$.
We consider two cases.

In case (i),  we assume that each vertex $u_j$ is joined  to
$v_j$ for every $0\leqslant j\leqslant n-1$.   From  $N(u_j)\cap N(u_{j+1})=\{v_j,v_{j+1}\}$, we have
$\lambda(u_j,u_{j+1})=\sigma(u_jv_j)-\sigma(u_{j+1}v_{j+1})=0.$
So $\sigma(u_jv_j)=\sigma(u_{j+1}v_{j+1})$. It follows that  all the edges $u_jv_j$, $0\leqslant j\leqslant n-1$ have the same sign.
 If these edges have sign $-1$, then
do switching at all vertices $u_j$, $0\leqslant j\leqslant n-1$ and after that relabel each vertex $u_j$ by $u_{n-j}$ and similarly $v_j$ with $v_{n-j}$
 to get $\Gamma=T^+_{n}$.
Since we are in Case 3, the cases $n=3,4$ are excluded, so $n\geqslant 5$.

It remains to consider case (ii) in which we assume that  some vertex $u_{\ell}$ is joined  to
some other vertex $u_k$ such that   $k> \ell$ and $k-\ell$ is minimum. Here, we have used Remark \ref{remT2n}.
We have $u_{\ell}\in N(u_{\ell+1})\cap N(u_{k})$. The orthogonality of $\Gamma$ and the minimality of $k-\ell$ force that $u_{k+1}\in N(u_{\ell+1})\cap N(u_k)$ or
$v_{k+1}\in N(u_{\ell+1})\cap N(u_k)$. The same argument applies to $N(v_{\ell+1})\cap N(u_{k})$.
This shows that there is a matching of size two from the vertex set $\{u_{\ell+1},v_{\ell+1}\}$ to the vertex set $\{u_{k+1},v_{k+1}\}$.
Repeating the argument, we deduce  that  there is a  matching of size two from $\{u_j,v_j\}$ to  $\{u_{j+k-\ell},v_{j+k-\ell}\}$ for every $0\leqslant j \leqslant n-1$.
Since $\Gamma$ is 5-regular, we necessarily  have $k-\ell=n/2$. By Remark  \ref{remT2n},  we  may  assume that   $u_j$ is joined to $v_{j+n/2}$ and $v_j$ is joined to
$u_{j+n/2}$ for every $0\leqslant j \leqslant n/2-2$. The same conclusion holds for $j=n/2-1$. However, to prove it, we need to determine the sign of the edges in already known matchings of size two.

Let $0\leqslant j \leqslant n/2-3$.
From  $N(u_j)\cap N(v_{j+1+n/2})=\{u_{j+1},v_{j+n/2}\}$, we have
$\lambda(u_j,v_{j+1+n/2})=\sigma(u_{j+1}v_{j+1+n/2})-\sigma(u_{j}v_{j+n/2})=0.$
So  $\sigma(u_{j+1}v_{j+1+n/2})=\sigma(u_{j}v_{j+n/2})$.
 Similar arguments apply to the other  edges in the matchings of size 2. It follows that the edges $u_jv_{j+n/2}$ and
$v_ju_{j+n/2}$ have the same sign, say $s$, for all $0\leqslant j \leqslant n/2-2$.
Now we consider the matching corresponding to $j=n/2-1$. Towards a contradiction,  assume that  $u_{n/2-1}$ is joined to $u_{n-1}$.
From  $N(u_{n-2})\cap N(u_{n/2-1})=\{u_{n-1},v_{n/2-2}\}$, we have
$\lambda(u_{n-2},u_{n/2-1})=\sigma(u_{n-1}u_{n/2-1})-s=0.$
So  $\sigma(u_{n-1}u_{n/2-1})=s$.
On the other hand,
From  $N(u_0)\cap N(u_{n/2-1})=\{u_{n-1},v_{n/2}\}$, one obtains
$\lambda(u_{0},u_{n/2-1})=\sigma(u_{n-1}u_{n/2-1})+s=0.$
So $\sigma(u_{n-1}u_{n/2-1})=-s$, a contradiction.

It remains to show that $s=1$. Suppose otherwise.  First switch at each vertex $u_j$, $0\leqslant j\leqslant n-1$. Then relabel  each vertex $u_j$ by $u_{n-j}$
  and  $v_j$ by $v_{n-j}$.
The resulting graph is the same as the original one except that the value of $s$ changes from $-1$ to 1.
Hence, $n=2n'$ is even and $\Gamma=T_{2n'}^{\ast}$.
\end{proof} 

\section{Some thoughts on $k>5$}\label{Sec:Kgeq6}

A natural next step would be to characterize the 6-regular orthogonal signed graphs. We do not attempt a full characterization here, as there appear to be a great number of disparate and complicated infinite families which must be accounted for. 

Instead, we mention a few of the infinite families we have found, and share some thoughts on how one might structure a program of classification.

\subsection{Decompositions into orthogonal ensembles} 
Let us say that a signed graph is an \textit{orthogonal ensemble} if it is a disjoint union of orthogonal signed graphs of the same valency. It is possible (though not necessary) that the edge set of an orthogonal signed graph admits a decomposition into a pair of orthogonal ensembles. 

\begin{definition}
A signed graph $\Gamma(A)$ is $(b,c)$-\emph{decomposable} if
$A = B + C$ where $B$ and $C$ are nonzero with disjoint support and each is the signed adjacency matrix of a $b$- (resp $c$-) regular orthogonal ensemble.
\end{definition}

The reason for this definition is that the vast majority of orthogonal signed graphs we have encountered are decomposable. Indeed, all three infinite families of 5-regular signed graphs are $(4,1)$-decomposable. In general, $(k,1)$-decomposability has a particularly clean formulation as the existence, in the $k$-regular part, of a fixed-point-free involutive automorphism of the base graph which reverses the signature. $(k,2)$-decomposability has a similar, but more complicated characterization. We spare the details but note two interesting infinite families of 6-regular signed graphs that we have found, each of which is $(4,2)$-decomposable.

\begin{definition}\label{defF8n}
We define the  family of orthogonal $6$-regular signed graphs $F_{4n}$ for $n\geqslant 2$. The vertex set is $V=\{0,1\}\times\{0,1,\ldots,4n-1\}$.
We start with  $F_{4n}=T_{4n}$ and  then  we  add the following edges, where all computations are modulo $4n$.
\begin{itemize}
\item[$\bullet$] $\{(0,0),(0,n)\}$,  $\{(1,0),(1,n)\}$  with sign $1$.
\item[$\bullet$] $\{(0,j),(1,j+n)\}$,  $\{(1,j),(0,j+n)\}$  with sign $-1$ for $0<j<n$.
\item[$\bullet$] $\{(0,n),(0,2n)\}$,  $\{(1,n),(1,2n)\}$  with sign $-1$.
\item[$\bullet$] $\{(0,j),(1,j+n)\}$,  $\{(1,j),(0,j+n)\}$  with sign $1$ for $n<j<4n$.
\end{itemize}
It is tedious but straightforward to check that $F_{4n}$ is in fact orthogonal.
\end{definition}

\begin{definition}\label{defU8n}
We define the  family of orthogonal $6$-regular signed graphs $U_{4n}$ for $n\geqslant 3$.
We start with  $U_{4n}=T_{2n}\cup T_{n}\cup T_{n}$ assuming   that
the vertex set is $V=V_O\cup V_L\cup V_R$ where $V_O=\{(i,j) | i=0,1, j=0,1,\ldots,2n-1\}$,
$V_L=\{(i,j)_L | i=0,1, j=0,1,\ldots,n-1\}$ and $V_R=\{(i,j)_R | i=0,1, j=0,1,\ldots,n-1\}$.
The induced signed subgraphs of $U_{4n}$ on $V_O$,  $V_L$ and $V_R$, respectively, are  $T_{2n}$, $T_{n}$ and $T_{n}$, respectively.
We  add the following edges.
\begin{itemize}
\item[$\bullet$] For $0\leqslant j<n$,  $\{(0,j),(0,j )_L\}$,  $\{(1,j),(1,j)_L\}$  with sign $1$ if $j$ is even and sign  $-1$ otherwise.
\item[$\bullet$] For $n\leqslant j<2n$,  $\{(0,j),(0,j-n)_L\}$,  $\{(1,j),(1,j-n)_L\}$  with sign $1$ if $j$ is even and sign  $-1$ otherwise.
\item[$\bullet$] $\{(0,0),(1,0)_R\}$,  $\{(1,0),(0,0)_R\}$  with sign $-1$.
\item[$\bullet$] For $0<j<n$,  $\{(0,j),(0,j)_R\}$,  $\{(1,j),(1,j)_R\}$  with sign $1$ if $j$ is even and sign  $-1$ otherwise.
\item[$\bullet$] $\{(0,n),(1,0)_R\}$,  $\{(1,n),(0,n)_R\}$  with sign $1$ if $n$ is even and sign  $-1$ otherwise.
\item[$\bullet$] For $n<j<2n$,  $\{(0,j),(0,j-n)_R\}$,  $\{(1,j),(1,j-n)_R\}$  with sign $1$ if $j$ is odd and sign  $-1$ otherwise.
\end{itemize}
Again, it is elementary to verify that $U_{4n}$ is in fact orthogonal.
\end{definition}

With these examples in mind, we think it may be prudent to organize the search for $6$-regular signed graphs into a search for lower valency decomposable orthogonal ensembles as well as a study of the properties of indecomposable orthogonal signed graphs. We leave this as future work. We do not expect either part of this program to be easy. In the next section we will construct an infinite family of indecomposable orthogonal 6-regular signed graphs.

\subsection{Splitting and joining}

In the $F_{4n}$ and $U_{4n}$ families above, any pair of vertices have at most 4 common neighbors. Next we present a notable 6-regular example which possesses pairs of vertices with 6 common neighbors. The basis of our construction is a new signature for the underlying graph of $T_6$ which can then be extended by adding two pairs of `hub' vertices attached to each of the two 6-cycles. Let $\eps(j)=(-1)^j$. 

\begin{definition}\label{defH6} (cf.~Figure \ref{fig:W_6}).
We define an orthogonal $6$-regular signed graph $W_6$ on the vertex set $V=\bigl(\{0,1\}\times\{0,1,\dots,5\}\bigr)\cup\{a,b,c,d\}$.  All computations
are modulo $6$. Each vertex $(0,j)$ is joined to $(0,j+1)$ with sign $-\eps(j)$, and each
vertex $(1,j)$ is joined to $(1,j+1)$ with sign $\eps(j)$.  Each vertex
$(0,j)$ is further joined to $(1,j+1)$ with sign $\eps(j)$ and to $(1,j-1)$
with sign $-\eps(j)$.  Finally the four \emph{hubs} are attached: $a$ and $b$
are each joined to every $(0,j)$, with signs $+1$ and $\eps(j)$ respectively,
while $c$ and $d$ are each joined to every $(1,j)$, with signs $+1$ and
$\eps(j)$ respectively. Every vertex has degree $6$, and $\{a,b\}$ and $\{c,d\}$ are called \textit{hub pairs}. Orthogonality is tedious but straightforward to verify.
\end{definition}

\begin{figure}
    \centering
    
\begin{center}
\begin{tikzpicture}[scale=1.15]
    \node[vtx] (v0) at (0.000,2.000) {};
    \node[vtx] (v1) at (-1.732,1.000) {};
    \node[vtx] (v2) at (-1.732,-1.000) {};
    \node[vtx] (v3) at (-0.000,-2.000) {};
    \node[vtx] (v4) at (1.732,-1.000) {};
    \node[vtx] (v5) at (1.732,1.000) {};
    \node[vtx] (v6) at (0.000,3.300) {};
    \node[vtx] (v7) at (-2.858,1.650) {};
    \node[vtx] (v8) at (-2.858,-1.650) {};
    \node[vtx] (v9) at (-0.000,-3.300) {};
    \node[vtx] (v10) at (2.858,-1.650) {};
    \node[vtx] (v11) at (2.858,1.650) {};
    \node[hubin] (v12) at (-0.550,0.320) {};
    \node[hubin] (v13) at (0.550,0.320) {};
    \node[hubout] (v14) at (-0.550,-0.320) {};
    \node[hubout] (v15) at (0.550,-0.320) {};
    \draw[en] (v0) -- (v1);
    \draw[ep] (v0) -- (v5);
    \draw[ep] (v0) -- (v7);
    \draw[en] (v0) -- (v11);
    \draw[ep] (v0) -- (v12);
    \draw[ep] (v0) -- (v13);
    \draw[ep] (v1) -- (v2);
    \draw[ep] (v1) -- (v6);
    \draw[en] (v1) -- (v8);
    \draw[ep] (v1) -- (v12);
    \draw[en] (v1) -- (v13);
    \draw[en] (v2) -- (v3);
    \draw[en] (v2) -- (v7);
    \draw[ep] (v2) -- (v9);
    \draw[ep] (v2) -- (v12);
    \draw[ep] (v2) -- (v13);
    \draw[ep] (v3) -- (v4);
    \draw[ep] (v3) -- (v8);
    \draw[en] (v3) -- (v10);
    \draw[ep] (v3) -- (v12);
    \draw[en] (v3) -- (v13);
    \draw[en] (v4) -- (v5);
    \draw[en] (v4) -- (v9);
    \draw[ep] (v4) -- (v11);
    \draw[ep] (v4) -- (v12);
    \draw[ep] (v4) -- (v13);
    \draw[en] (v5) -- (v6);
    \draw[ep] (v5) -- (v10);
    \draw[ep] (v5) -- (v12);
    \draw[en] (v5) -- (v13);
    \draw[ep] (v6) -- (v7);
    \draw[en] (v6) -- (v11);
    \draw[en] (v7) -- (v8);
    \draw[ep] (v8) -- (v9);
    \draw[en] (v9) -- (v10);
    \draw[ep] (v10) -- (v11);
\end{tikzpicture}
\end{center}
    \caption{The graph $W_6$. The outer hub pair has been deleted to make the figure more readable.}
    \label{fig:W_6}
\end{figure}
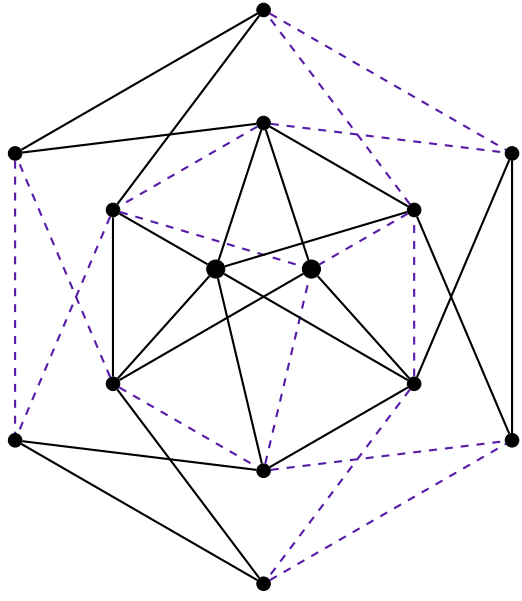

This example is remarkable in that it can be extended to an infinite family of 6-regular orthogonal signed graphs as follows (cf. Figure \ref{fig:Split_Join}):

\begin{figure}[ht]
\centering
\resizebox{\textwidth}{!}{%
\definecolor{graphgray}{RGB}{70,82,82}
\definecolor{accent}{RGB}{139,63,214}
\definecolor{ringred}{RGB}{235,145,145}
\definecolor{ringblue}{RGB}{115,180,245}

\tikzset{
  graph edge/.style={draw=graphgray, line width=0.95pt, line cap=round, line join=round},
  hi edge/.style={draw=accent, dashed, dash pattern=on 4pt off 2.5pt, line width=1.15pt, line cap=round, line join=round},
  boundary/.style={draw=graphgray, line width=0.95pt},
  point/.style={circle, fill=graphgray, inner sep=1.1pt},
  lab/.style={font=\normalsize, text=graphgray, inner sep=1pt},
  small lab/.style={font=\small, text=graphgray, fill=white, fill opacity=.94, text opacity=1, inner sep=1.0pt}
}

\begin{tikzpicture}[x=0.72cm,y=0.72cm]

\begin{scope}[shift={(0.0,0.0)}]
  \draw[boundary] (1.5,2.6) ellipse [x radius=1.0, y radius=2.3];
  \draw[boundary] (7.2,2.6) ellipse [x radius=1.0, y radius=2.3];

  \coordinate (LP1) at (1.8,4.1);
  \coordinate (LP2) at (1.8,3.5);
  \coordinate (LP3) at (1.8,2.9);
  \coordinate (LM1) at (1.8,2.0);
  \coordinate (LM2) at (1.8,1.4);
  \coordinate (LM3) at (1.8,0.8);

  \coordinate (LA) at (3.7,3.7);
  \coordinate (LB) at (3.7,1.7);
  \coordinate (RA) at (5.0,3.7);
  \coordinate (RB) at (5.0,1.7);

  \coordinate (RM3) at (6.9,4.1);
  \coordinate (RM2) at (6.9,3.5);
  \coordinate (RM1) at (6.9,2.9);
  \coordinate (RP3) at (6.9,2.0);
  \coordinate (RP2) at (6.9,1.4);
  \coordinate (RP1) at (6.9,0.8);

  \foreach \x in {LP1,LP2,LP3,LM1,LM2,LM3}{\draw[graph edge] (\x) -- (LA);}
  \foreach \x in {LP1,LP2,LP3}{\draw[graph edge] (\x) -- (LB);}
  \foreach \x in {LM1,LM2,LM3}{\draw[hi edge] (\x) -- (LB);}

  \foreach \x in {RM3,RM2,RM1,RP3,RP2,RP1}{\draw[graph edge] (RB) -- (\x);}
  \foreach \x in {RP3,RP2,RP1}{\draw[graph edge] (RA) -- (\x);}
  \foreach \x in {RM3,RM2,RM1}{\draw[hi edge] (RA) -- (\x);}

  \node[lab,anchor=east] at (LP1) {$p_1$};
  \node[lab,anchor=east] at (LP2) {$p_2$};
  \node[lab,anchor=east] at (LP3) {$p_3$};
  \node[lab,anchor=east] at (LM1) {$m_1$};
  \node[lab,anchor=east] at (LM2) {$m_2$};
  \node[lab,anchor=east] at (LM3) {$m_3$};

  \node[lab,anchor=west] at (RM3) {$m_3'$};
  \node[lab,anchor=west] at (RM2) {$m_2'$};
  \node[lab,anchor=west] at (RM1) {$m_1'$};
  \node[lab,anchor=west] at (RP3) {$p_3'$};
  \node[lab,anchor=west] at (RP2) {$p_2'$};
  \node[lab,anchor=west] at (RP1) {$p_1'$};

  \foreach \x in {LA,LB,RA,RB}{\node[point] at (\x) {};}
  \node[lab,anchor=west] at ($(LA)+(0.1,0.1)$) {$a$};
  \node[lab,anchor=west] at ($(LB)+(0.1,-0.1)$) {$b$};
  \node[lab,anchor=east] at ($(RA)+(-0.1,0.1)$) {$b'$};
  \node[lab,anchor=east] at ($(RB)+(-0.1,-0.1)$) {$a'$};
\end{scope}

\begin{scope}[shift={(10.2,0.0)}]
  \draw[boundary] (1.5,2.6) ellipse [x radius=1.0, y radius=2.3];
  \draw[boundary] (8.9,2.6) ellipse [x radius=1.0, y radius=2.3];

  \coordinate (BP1) at (1.8,4.1);
  \coordinate (BP2) at (1.8,3.5);
  \coordinate (BP3) at (1.8,2.9);
  \coordinate (BM1) at (1.8,2.0);
  \coordinate (BM2) at (1.8,1.4);
  \coordinate (BM3) at (1.8,0.8);

  \coordinate (CM1) at (8.6,4.1);
  \coordinate (CM2) at (8.6,3.5);
  \coordinate (CM3) at (8.6,2.9);
  \coordinate (CP1) at (8.6,2.0);
  \coordinate (CP2) at (8.6,1.4);
  \coordinate (CP3) at (8.6,0.8);

  \coordinate (A) at (4.8,4.0);
  \coordinate (B) at (5.4,3.3);
  \coordinate (C) at (4.8,2.1);
  \coordinate (D) at (5.4,1.4);

  \foreach \x in {BP1,BP2,BP3}{
    \draw[graph edge] (\x) -- (A);
    \draw[graph edge] (\x) -- (B);
  }
  \foreach \x in {CM1,CM2,CM3}{
    \draw[hi edge] (A) -- (\x);
    \draw[graph edge] (B) -- (\x);
  }

  \foreach \x in {BM1,BM2,BM3}{
    \draw[graph edge] (\x) -- (C);
    \draw[hi edge] (\x) -- (D);
  }
  \foreach \x in {CP1,CP2,CP3}{
    \draw[graph edge] (C) -- (\x);
    \draw[graph edge] (D) -- (\x);
  }

  \node[lab,anchor=east] at (BP1) {$p_1$};
  \node[lab,anchor=east] at (BP2) {$p_2$};
  \node[lab,anchor=east] at (BP3) {$p_3$};
  \node[lab,anchor=east] at (BM1) {$m_1$};
  \node[lab,anchor=east] at (BM2) {$m_2$};
  \node[lab,anchor=east] at (BM3) {$m_3$};

  \node[lab,anchor=west] at (CM1) {$m_3'$};
  \node[lab,anchor=west] at (CM2) {$m_2'$};
  \node[lab,anchor=west] at (CM3) {$m_1'$};
  \node[lab,anchor=west] at (CP1) {$p'_3$};
  \node[lab,anchor=west] at (CP2) {$p'_2$};
  \node[lab,anchor=west] at (CP3) {$p'_1$};

  \foreach \x in {A,B,C,D}{\node[point] at (\x) {};}

  \node[small lab,anchor=south east] at ($(A)+(-0.1,0.1)$) {$a_1$};
  \node[small lab,anchor=south west] at ($(A)+(0.1,0.1)$) {$b'_1$};

  \node[small lab,anchor=north east] at ($(B)+(-0.1,-0.1)$) {$a_2$};
  \node[small lab,anchor=north west] at ($(B)+(0.1,-0.1)$) {$b'_2$};

  \node[small lab,anchor=south east] at ($(C)+(-0.1,0.1)$) {$b_1$};
  \node[small lab,anchor=south west] at ($(C)+(0.1,0.1)$) {$a'_2$};

  \node[small lab,anchor=north east] at ($(D)+(-0.1,-0.1)$) {$b_2$};
  \node[small lab,anchor=north west] at ($(D)+(0.1,-0.1)$) {$a'_1$};
\end{scope}

\end{tikzpicture}%
}

    \caption{Splitting and joining two hub pairs $\{a,b\}$ and $\{a',b'\}$. The additional labels are used below. }
    \label{fig:Split_Join}
\end{figure}
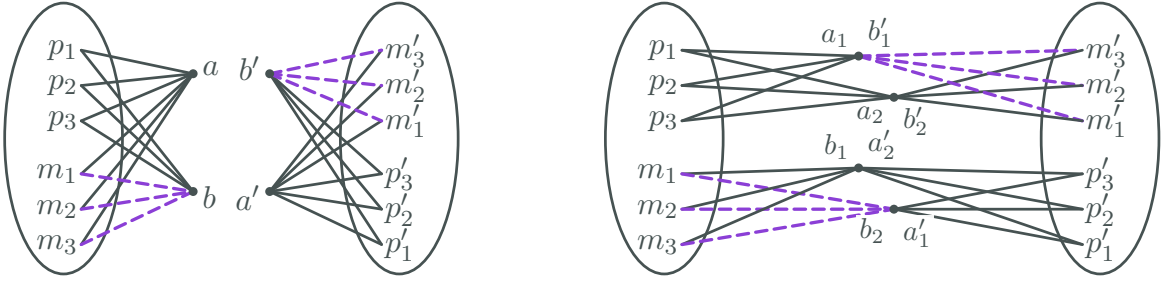

First we `split' the vertices $a$ and $b$ each into two vertices of degree three, $a_1,a_2$ and $b_1,b_2$. We do this so that $b_1$ and $b_2$ are each incident with three edges of the same sign. Call this graph $W_6^s$ and take a second copy of $W_6^s$ with vertices $v'$ for $v\in V(W_6^s)$ and identify the following pairs of vertices: $(a_1,b_1'),(a_2,b_2'),(b_1,a_2'),(b_2,a_1')$. It is easy to check by counting 2-walks that this 6-regular signed graph is orthogonal. We may continue splitting and joining to add as many copies of $W_6$ as desired, and we may join the first and final ends as well.

Now we formulate this construction in a bit more detail and generality. Let $\Gamma$ be a $2\ell$-regular orthogonal signed graph which possesses distinct vertices $a,b$ such that $|N_a\cap N_b|=2\ell$. Let $P=\{p_1,\dots ,p_\ell\}$ be the subset of $V(\Gamma)$ consisting of  midpoints of the positive $2$-walks from $a$ to $b$ and let  $M=\{m_1,\dots, m_\ell\}$ be the subset of $V(\Gamma)$ consisting of midpoints of negative $2$-walks from $a$ to $b$. We say that such a signed graph is {\sl splittable at $a,b$} if $a$ and $b$ are the only vertices of $\Gamma$ that have neighbors in both $P$ and $M$. $\Gamma$ is {\sl splittable} if it is splittable at some pair of vertices $a,b$. 

Let $\Gamma_1$ be splittable at $a,b$ and $\Gamma_2$ be splittable at $a',b'$. The {\sl split and join of $\Gamma_1$ and $\Gamma_2$ at $a,b,a',b'$} is a signed graph denoted $\Gamma_1^{a,b}*\Gamma_2^{a',b'}$ and defined as follows (see Figure \ref{fig:Split_Join}).

Delete $a,b,a',b'$ and their incident edges and create new vertices $a_1,a_2,b_1,b_2$ with adjacency as follows: 
\begin{itemize}
    \item $N(a_1)=N(a_2)=\{p_1,\dots, p_\ell,m_1',\dots ,m_\ell'\}$
    
    \item $N(b_1)=N(b_2)=\{p_1',\dots, p_\ell',m_1,\dots, m_\ell\}$

    \item The edges $\{a,m_i'\}$ and $\{d,m_i\}$ are negative, and all other new edges are positive.
\end{itemize}

\begin{lemma}\label{Thm:Splitjoin}
Let $\Gamma_1$ and $\Gamma_2$ be $2\ell-$regular orthogonal signed graphs. Suppose $\Gamma_1$ is splittable at $a,b$ and $\Gamma_2$ is splittable at $a',b'$. If $\Gamma_1$ is $3$-connected then $\Gamma_1^{a,b}*\Gamma_2^{a',b'}$ is a $2\ell-$regular orthogonal signed graph.
\end{lemma}

\begin{proof}
We only need to check that the 2-walks between each pair of vertices $x,y\in V(\Gamma_1^{a,b}*\Gamma_2^{a',b'})$ remain balanced. Clearly this is true if neither $x$ nor $y$ belongs to the set $S=\{a_1,a_2,b_1,b_2\}\cup P\cup M\cup P'\cup M'$ since no relevant adjacency or sign on a 2-walk between $x$ and $y$ has been altered.

Suppose $x\in S,y\notin S$. Again, if $x\in  P\cup M\cup P'\cup M'$ nothing is altered and there is nothing to show. If $x\in \{a_1,a_2,b_1,b_2\}$ we may assume without loss that $y\in V(\Gamma_1)$, now since $\Gamma_1$ is splittable, $y$ is not incident with both $P$ and $M$, so none of the 2-walks from $x$ to $y$ have been altered in this case either. 

Now if $x,y\in S$ we have several cases to check. There are only a few interesting cases, up to symmetry. For $x\in P,y\in M'$, two paths between the vertices are added, but they are balanced by construction. For $x\in P,y\in M$, two paths are destroyed, but they were balanced in $\Gamma_1$ or $\Gamma_2$, so $x$ and $y$ remain balanced. Finally, $\Gamma_1^{a,b}*\Gamma_2^{a',b'}$ is connected since $\Gamma_1$ is $3$-connected.
\end{proof}

Two remarks on the split and join operation are in order. 

First, note that the 4-regular orthogonal signed graphs $T_n$ are themselves splittable, but the application of the split and join operation does not give anything new. In particular, let $\Gamma_1$ be $T_{n_1}$ and pick vertices $u,v$ with four common neighbors. Let $\Gamma_2$ be $T_{n_2}$ and pick vertices $u',v'$ with four common neighbors. It is easy to see that these signed graphs are splittable at these vertices, and that $\Gamma_1^{u,v}*\Gamma_2^{u',v'}\cong T_{n_1+n_2}$.

Second, note that a pair of the joined vertices of $W_6*W_6$ is itself splittable. This allows one to construct many infinite families from iterated splitting of $W_6*W_6$, not just the one mentioned above. 

To conclude, we offer an example of an 8-regular orthogonal signed graph $W_8$ which possesses two pairs of splittable vertices. It is essentially a generalization of $W_6$ with an extra set of diagonal edges between the $n$-cycles. 
Using Theorem \ref{Thm:Splitjoin} iteratively, we may extend $W_8$ to a new infinite family of $8$-regular orthogonal signed graphs.

\begin{definition}\label{defH8} (cf.~Figure \ref{fig:W8}).
We define the orthogonal $8$-regular signed graph $W_8$.  The vertex set is
$V=\bigl(\{0,1\}\times\{0,1,\dots,7\}\bigr)\cup\{a,b,c,d\}$.  All computations
are modulo $8$.

Each vertex $(0,j)$ is joined to $(0,j+1)$, and each vertex $(1,j)$ to
$(1,j+1)$, both with sign $-\eps(j)$.  Each vertex $(0,j)$ is further joined to
four outer vertices in two diagonal pairs: to $(1,j+1)$ with sign $-1$ and to
$(1,j-1)$ with sign $+1$ (the odd offset, constant sign); and to $(1,j+2)$ with
sign $\eps(j)$ and to $(1,j-2)$ with sign $-\eps(j)$ (the even offset,
alternating sign).  Finally the four \emph{hubs}: $a$ and $b$ are each joined
to every $(0,j)$, with signs $+1$ and $\eps(j)$; while $c$ and $d$ are each
joined to every $(1,j)$, with signs $+1$ and $-\eps(j)$. Orthogonality is tedious but straightforward to verify.
\end{definition}

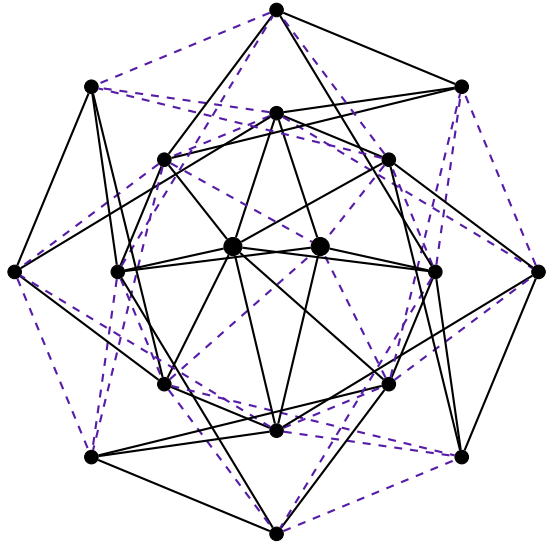
\begin{figure}
    \centering
    
\begin{center}
\begin{tikzpicture}[scale=1.05]
    \node[vtx] (v0) at (0.000,2.000) {};
    \node[vtx] (v1) at (-1.414,1.414) {};
    \node[vtx] (v2) at (-2.000,0.000) {};
    \node[vtx] (v3) at (-1.414,-1.414) {};
    \node[vtx] (v4) at (-0.000,-2.000) {};
    \node[vtx] (v5) at (1.414,-1.414) {};
    \node[vtx] (v6) at (2.000,-0.000) {};
    \node[vtx] (v7) at (1.414,1.414) {};
    \node[vtx] (v8) at (0.000,3.300) {};
    \node[vtx] (v9) at (-2.333,2.333) {};
    \node[vtx] (v10) at (-3.300,0.000) {};
    \node[vtx] (v11) at (-2.333,-2.333) {};
    \node[vtx] (v12) at (-0.000,-3.300) {};
    \node[vtx] (v13) at (2.333,-2.333) {};
    \node[vtx] (v14) at (3.300,-0.000) {};
    \node[vtx] (v15) at (2.333,2.333) {};
    \node[hubin] (v16) at (-0.550,0.320) {};
    \node[hubin] (v17) at (0.550,0.320) {};
    \node[hubout] (v18) at (-0.550,-0.320) {};
    \node[hubout] (v19) at (0.550,-0.320) {};
    \draw[en] (v0) -- (v1);
    \draw[ep] (v0) -- (v7);
    \draw[en] (v0) -- (v9);
    \draw[ep] (v0) -- (v10);
    \draw[en] (v0) -- (v14);
    \draw[ep] (v0) -- (v15);
    \draw[ep] (v0) -- (v16);
    \draw[ep] (v0) -- (v17);
    \draw[ep] (v1) -- (v2);
    \draw[ep] (v1) -- (v8);
    \draw[en] (v1) -- (v10);
    \draw[en] (v1) -- (v11);
    \draw[ep] (v1) -- (v15);
    \draw[ep] (v1) -- (v16);
    \draw[en] (v1) -- (v17);
    \draw[en] (v2) -- (v3);
    \draw[en] (v2) -- (v8);
    \draw[ep] (v2) -- (v9);
    \draw[en] (v2) -- (v11);
    \draw[ep] (v2) -- (v12);
    \draw[ep] (v2) -- (v16);
    \draw[ep] (v2) -- (v17);
    \draw[ep] (v3) -- (v4);
    \draw[ep] (v3) -- (v9);
    \draw[ep] (v3) -- (v10);
    \draw[en] (v3) -- (v12);
    \draw[en] (v3) -- (v13);
    \draw[ep] (v3) -- (v16);
    \draw[en] (v3) -- (v17);
    \draw[en] (v4) -- (v5);
    \draw[en] (v4) -- (v10);
    \draw[ep] (v4) -- (v11);
    \draw[en] (v4) -- (v13);
    \draw[ep] (v4) -- (v14);
    \draw[ep] (v4) -- (v16);
    \draw[ep] (v4) -- (v17);
    \draw[ep] (v5) -- (v6);
    \draw[ep] (v5) -- (v11);
    \draw[ep] (v5) -- (v12);
    \draw[en] (v5) -- (v14);
    \draw[en] (v5) -- (v15);
    \draw[ep] (v5) -- (v16);
    \draw[en] (v5) -- (v17);
    \draw[en] (v6) -- (v7);
    \draw[ep] (v6) -- (v8);
    \draw[en] (v6) -- (v12);
    \draw[ep] (v6) -- (v13);
    \draw[en] (v6) -- (v15);
    \draw[ep] (v6) -- (v16);
    \draw[ep] (v6) -- (v17);
    \draw[en] (v7) -- (v8);
    \draw[en] (v7) -- (v9);
    \draw[ep] (v7) -- (v13);
    \draw[ep] (v7) -- (v14);
    \draw[ep] (v7) -- (v16);
    \draw[en] (v7) -- (v17);
    \draw[en] (v8) -- (v9);
    \draw[ep] (v8) -- (v15);
    \draw[ep] (v9) -- (v10);
    \draw[en] (v10) -- (v11);
    \draw[ep] (v11) -- (v12);
    \draw[en] (v12) -- (v13);
    \draw[ep] (v13) -- (v14);
    \draw[en] (v14) -- (v15);
\end{tikzpicture}
\end{center}
    \caption{The 8-regular orthogonal signed graph $W_8$. The outer hub pair has been deleted to make the figure more readable.}
    \label{fig:W8}
\end{figure}

\section*{Statement of AI usage.} This paper was almost entirely human-generated. ChatGPT 5.6 Sol and Claude Opus 5 were used to generate the tikz code for the figures and for proofreading. The signed graph $W_8$ was found using an exhaustive search designed and written by Claude. We gave Claude our description of $W_6$ (which was found by hand), and asked it to extend the construction to larger degrees. The authors take full responsibility for the accuracy and correctness of all statements made within the paper. 

\bibliographystyle{amsplain}
\bibliography{Orth_Sig_BCM_no_DOI}

\end{document}